\documentclass{amsart}

\usepackage{amsmath}
\usepackage{amssymb}
\usepackage{amsfonts}
\usepackage{amsbsy}
\usepackage{amscd}
\usepackage{eucal}
\usepackage{graphicx}
\usepackage{caption}

\DeclareMathOperator{\sgn}{sgn}

\newtheorem{theorem}{Theorem}
\newtheorem{proposition}[theorem]{Proposition}
\newtheorem{lemma}[theorem]{Lemma}

\theoremstyle{remark}
\newtheorem{remark}{Remark}

\newcommand{\beq}{\begin{equation}}
\newcommand{\eeq}{\end{equation}}

\newcommand{\rmd}{\mathrm{d}}
\newcommand{\rmi}{\mathrm{i}}
\newcommand{\rmI}{\mathrm{I}}

\newcommand{\R}{\mathbb{R}}
\newcommand{\Z}{\mathbb{Z}}

\newcommand{\txs}{\textstyle}

\newcommand{\oP}{\overline{P}}

\newcommand{\ud}{\frac{1}{2}}

\newcommand{\skd}{\vspace*{0.2cm}}

\begin{document}

\title[A fast algorithm for inversion of Abel's transform]
{A fast algorithm for the inversion of Abel's transform}

\author{Enrico De Micheli}
\address{IBF - Consiglio Nazionale delle Ricerche,
Via De Marini, 6 - 16149 Genova, Italy}
\email{enrico.demicheli@cnr.it}

\subjclass[2010]{45E10,44A15,65R32}

\keywords{Abel inversion, Legendre polynomials, Radio occultation,
Plasma emission coefficients, Inverse problems, Stability estimates}

\begin{abstract} 
We present a new algorithm for the computation of the inverse Abel transform, a problem which
emerges in many areas of physics and engineering.
We prove that the Legendre coefficients of a given function coincide with the 
Fourier coefficients of a suitable periodic 
function associated with its Abel transform. This allows us to compute the Legendre coefficients 
of the inverse Abel transform in an easy, fast and accurate way by means of a single Fast Fourier Transform.
The algorithm is thus appropriate also for the inversion of Abel integrals given in terms 
of samples representing noisy measurements. Rigorous stability estimates are proved
and the accuracy of the algorithm is illustrated also by some numerical experiments.
\end{abstract}


\maketitle

\section{Introduction}
\label{se:introduction}

The subject of this paper is the analysis and the numerical solution of the Abel integral equation of the first kind:
\beq
g(x) = (Af)(x) \doteq \int_{0}^x \frac{f(y)}{\sqrt{x-y}}\,\rmd y \qquad (0\leqslant x\leqslant 1).
\label{1.1}
\eeq
In \eqref{1.1}, $g(x)\in L^2(0,1)$ represents the known data function, and $f(x)\in L^2(0,1)$ is the unknown 
function to be computed. We can assume, with no loss of generality, $g(0)=0$. Therefore, equation
\eqref{1.1} defines a linear compact operator $A:L^2(0,1) \to L^2(0,1)$ \cite{Gorenflo1}.
Abel's integral equation plays an important role in many areas of science. Its most extensive use is for 
the determination of the radial distribution of cylindrically symmetric physical quantities, e.g.
the plasma emission coefficients, from line-of-sight 
integration measurements. In X-ray tomography, the object being analyzed is illuminated by parallel X-ray beams 
and an Abel equation of type \eqref{1.1} relates the intensity profile of the transmitted rays 
(the data function $g$) to the object's radial density profile (the unknown function $f$) \cite{Abramson,Asaki2}.
Abel inversion is widely used in plasma physics to obtain the 
electronic density from phase-shift maps obtained by laser interferometry \cite{Park}
or radial emission patterns from observed plasma radiances \cite{Fasoulas,Merk}. 
Photoion and photoelectron imaging in molecular dynamics \cite{Dribinski}, evaluation of mass density and velocity laws of stellar winds in 
astrophysics \cite{Craig,Knill}, and atmospheric radio occultation signal analysis \cite{Healy,Thomson}  
are additional fields which frequently require the numerical solution of Abel's equations of type \eqref{1.1}.

The exact solution to \eqref{1.1} traces back to Abel's memoir \cite{Abel} (see also \cite{Courant}):
\beq
f(y) = \frac{1}{\pi}\frac{\rmd}{\rmd y} \int_{0}^y \frac{g(x)}{\sqrt{y-x}}\,\rmd x
= \frac{1}{\pi}\int_{0}^y \frac{g'(x)}{\sqrt{y-x}}\,\rmd x \qquad (\mathrm{with} \ g(0)=0),
\label{1.2}
\eeq
and existence results are conveniently given 
in Ref. \cite{Gorenflo1} for pairs of functions $f$ and $g$ belonging 
to a variety of functional spaces (e.g., H\"older spaces and Lebesgue spaces).

Actual difficulties arise when the Abel inversion has to be computed from input data which are noisy
and finite in number, as when the data represent experimental measurements.
In this case Eq. \eqref{1.2} is often 
of little practical utility since it requires the numerical differentiation which tends to amplify the errors. 
The Abel inversion is in fact a (mildly) ill-posed problem since the solution does not depend continuously 
on the data: slight inaccuracies in the input data may lead to a solution very \emph{far} from the true 
one. Stated in other words, since the operator $A:L^2(0,1) \to L^2(0,1)$ is compact, 
its inverse $A^{-1}$ cannot be continuous in the $L^2$-norm \cite{Gorenflo1}. 
It is therefore crucial to set up numerical algorithms which yield a stable solution to problem \eqref{1.1}.
For these reasons, various numerical methods for the inversion of Abel's operator have been proposed.
In Refs. \cite{Deutsch,Fasoulas,Gueron}
input data are represented through cubic spline 
and then the inverse Abel transformation is applied to get the solution. 
Iterative schemes \cite{Vicharelli} have been proved 
to be rather stable but are time-consuming. 
Inversion techniques based on the Fourier-Hankel transform are discussed in Refs. \cite{Alvarez,Kalal}.
Numerical methods developed within regularization schemes 
suitable for problem \eqref{1.1} have also been presented \cite{Asaki1,Gorenflo2}.
The representation of input data and solution in various orthonormal basis in Hilbert spaces, coupled with the 
exact inversion of the Abel integral operator, has been exploited in refs. \cite{Dribinski,Garza,Li}.
The importance of using orthogonal polynomials for the stable solution of problem \eqref{1.1} has been recognized
for a long time \cite{Minerbo}. The approximation of the unknown solution by Jacobi polynomials \cite{Ammari1},
Legendre polynomials \cite{Ammari1,Ammari2} and Chebyshev polynomials \cite{Pandey,Sugiura}
has been proposed for the inversion of the Abel integral operator.

In this paper we present a new procedure for the computation of the inverse Abel transform.
In Section \ref{se:inversion} we prove that the Legendre coefficients of the solution $f(x)$ 
to problem \eqref{1.1} coincide with the Fourier coefficients 
of a suitable function associated with the data $g(x)$. The role of noise is studied in Section 
\ref{se:regularization} where we focus on the regularization of problem \eqref{1.1} within the spectral cut-off 
scheme and introduce a suitably regularized solution $f_N^{(\varepsilon)}(x)$, $\varepsilon$ being a parameter 
which represents the amount of noise on the data. 
Rigorous stability estimates for the proposed solution are then proved in the same
Section \ref{se:regularization}, where we give upper bounds on the reconstruction error which depend on the 
smoothness properties of the solution and on the level of noise $\varepsilon$. 
The algorithm produced by this analysis results to be extremely simple and fast since the $N$ coefficients of the 
regularized solution can be computed very efficiently by means of a single Fast Fourier Transform in
$\mathcal{O}(N\log N)$ time. This 
attractive feature makes the algorithm particularly suitable for the Abel inversion of 
functions represented by samples, e.g., noisy experimental measurements, given on nonequispaced points of 
the Abel transform domain since the core of the computation can be simply performed by means of a 
nonuniform fast Fourier transform.
Finally, in Section \ref{se:numerical} we illustrate some numerical experiments which exemplify the theoretical 
analysis and give a taste of the stability of the algorithm for the computation of the inverse Abel transform 
for solutions with different smoothness properties and various levels of noise on the data.

\section{Inversion of the Abel transform by Legendre expansion}
\label{se:inversion}

For convenience, let us define the following intervals of the real line: $E \doteq (0,1)$, 
$\Omega \doteq (-\pi,\pi)$.
Consider the \emph{shifted} Legendre polynomials $\oP_n(x)$, which are defined by:
\beq
\oP_n(x) = \sqrt{2n+1} \, P_n(2x-1),
\label{2.-1}
\eeq
where $P_n(x)$ denote the \emph{ordinary} Legendre polynomials, defined by the generating function \cite{Erdelyi2}:
\beq
\sum_{n=0}^\infty P_n(x)\,t^n = \left(1-2xt+t^2\right)^{-\ud}.
\label{2.0}
\eeq
The \emph{shifted} Legendre polynomials $\{\oP_n(x)\}_{n=0}^\infty$ form a complete orthonormal basis for $L^2(E)$.
The (\emph{shifted}) Legendre expansion of a function $f(x)\in L^2(E)$ reads: 
\beq
f(x) = \sum_{n=0}^\infty c_n \, \oP_n(x) \qquad (x \in E),
\label{2.1}
\eeq
with coefficients $c_n=(f,\oP_n)$ (where $(\cdot,\cdot)$ denotes the scalar product in $L^2(E)$):
\beq
c_n = \int_{0}^1 f(x)\,\oP_n(x)\,\rmd x \qquad (n\geqslant 0).
\label{2.2}
\eeq 

We can now prove the following theorem which connects the Legendre coefficients of a function $f(y)$ with its
Abel transform $g(x)$.

\begin{theorem}
\label{the:4}

Let $g$ denote the Abel transform \eqref{1.1} of the function $f\in L^2(E)$.
Then the inverse Abel transform $f=(A^{-1}g)$ can be written as:
\beq
f(x) = \sum_{n=0}^{\infty} c_n \,\oP_n(x) \qquad (x\in E),
\label{t.2}
\eeq
where $c_n=(-1)^n\sqrt{2n+1}\,\hat{\gamma}_n$ and the coefficients $\{\hat{\gamma}_n\}_{n=0}^\infty$
coincide with the Fourier coefficients (with $n\geqslant 0$) 
of the $2\pi$-periodic auxiliary function $\eta(t)$ $(t\in\R)$, whose restriction to the interval $t \in [-\pi,\pi)$ 
is given by
\beq
\eta(t) \doteq \frac{\sgn(t)}{2\pi\rmi}\,e^{\rmi t/2} \, g\left(\sin^2\frac{t}{2}\right) 
\qquad (t\in[-\pi,\pi)),
\label{t.4}
\eeq
where $\sgn(t)$ denotes the sign function, that is, we have:
\beq
\hat{\gamma}_n
\doteq\frac{(-1)^n}{\sqrt{2n+1}}\,c_n = \int_{-\pi}^\pi \eta(t)\,e^{\rmi n t}\,\rmd t \qquad (n \geqslant 0).
\label{t.3}
\eeq
\end{theorem}

\begin{proof} 
From \eqref{2.2}, using definition \eqref{2.-1}, and changing the variable of integration 
$x \to (1-\cos u)/2$ we have:
\beq
\begin{split}
c_n &= \sqrt{2n+1} \int_{0}^1 f(x)\,P_n(2x-1)\,\rmd x \\
&= \frac{(-1)^n \sqrt{2n+1}}{2} \int_{0}^\pi f\left(\frac{1-\cos u}{2}\right)\,P_n(\cos u)\,\sin u\,\rmd u.
\end{split}
\label{t.5}
\eeq
Plugging the Dirichlet-Murphy integral representation of the Legendre polynomials \cite[Ch. III, \S 5.4]{Vilenkin}:
\beq
P_n(\cos u)=-\frac{\rmi}{\pi\sqrt{2}}\int_u^{(2\pi-u)}
\frac{e^{\,\rmi(n+\ud)t}}{\sqrt{\cos u - \cos t}}\,\rmd t,
\label{t.6}
\eeq
into \eqref{t.5}, we obtain:
\beq
\frac{2\sqrt{2}\pi\rmi (-1)^n}{\sqrt{2n+1}} \, c_n
=\int_0^\pi \rmd u\,f\left(\sin^2\left(\frac{u}{2}\right)\right)\sin u\,
\int_u^{(2\pi-u)}\frac{e^{\,\rmi(n+\ud)t}}
{\sqrt{\cos u - \cos t}}\,\rmd t.
\label{t.7}
\eeq
Interchanging the order of integration in \eqref{t.7} we have:
\beq
\begin{split}
\frac{2\sqrt{2}\pi\rmi (-1)^n}{\sqrt{2n+1}} \, c_n
&=\int_0^\pi \rmd t \ e^{\,\rmi(n+\ud)t}\int_0^t f\left(\sin^2\left(\frac{u}{2}\right)\right)\,
\frac{\sin u}{\sqrt{\cos u - \cos t}}\,\rmd u \\
&\quad +\int_\pi^{2\pi} \rmd t \ e^{\,\rmi(n+\ud)t} 
\int_0^{(2\pi-t)} f\left(\sin^2\left(\frac{u}{2}\right)\right) \ \frac{\sin u}{\sqrt{\cos u - \cos t}}\, \rmd u.
\end{split}
\label{t.8}
\eeq
Next, if we make the change of variables: $t\to t-2\pi$ and $u \to -u$,
the second integral on the r.h.s. of \eqref{t.8} becomes: 
\beq
e^{\rmi\pi}\int_{-\pi}^0 \rmd t \ e^{\,\rmi(n+\ud)t}
\int_0^t f\left(\sin^2\left(\frac{u}{2}\right)\right) \ 
\frac{\sin u}{\sqrt{\cos u - \cos t}}\, \rmd u.
\label{t.9}
\eeq
Finally, we obtain:
\beq
\begin{split}
\frac{2\sqrt{2}\pi\rmi (-1)^n}{\sqrt{2n+1}} \, c_n
&=\int_0^\pi \rmd t \ e^{\,\rmi(n+\ud)t}
\int_0^t f\left(\sin^2\left(\frac{u}{2}\right)\right) \ \frac{\sin u}{\sqrt{\cos u - \cos t}}\, \rmd u \\
& \quad +e^{\rmi\pi}\int_{-\pi}^0 \rmd t \ e^{\,\rmi(n+\ud)t}
\int_0^t f\left(\sin^2\left(\frac{u}{2}\right)\right) \ \frac{\sin u}{\sqrt{\cos u - \cos t}}\, \rmd u \\
&=\int_{-\pi}^\pi \!\!\rmd t \ e^{\rmi nt} \, \left[\sgn(t)\,e^{\rmi t/2} \!\! 
\int_0^t f\left(\sin^2\left(\frac{u}{2}\right)\right) \, \frac{\sin u}{\sqrt{\cos u - \cos t}}\, \rmd u\right].
\end{split}
\label{t.10}
\eeq
Let $\hat{\gamma}_n$ denote the auxiliary Legendre coefficients 
$\hat{\gamma}_n \doteq (-1)^n\,c_n/\sqrt{2n+1}$;
then, from \eqref{t.10} we have:
\beq
\hat{\gamma}_n
=\int_{-\pi}^\pi \eta(t) \,e^{\rmi nt}\,\rmd t \qquad (n\geqslant 0),
\label{t.11}
\eeq
where $\eta(t)$ is the $2\pi$-periodization on the real line of the function defined 
on the interval $t\in[-\pi,\pi)$ by:
\beq
\begin{split}
\eta(t) =
\frac{\sgn(t)e^{\rmi\frac{t}{2}}}{2\pi\rmi}
\int_0^{\sin^2(\frac{t}{2})}\!\!\!\frac{f(y)\,\rmd y}{\sqrt{\sin^2(\frac{t}{2})-y}}
=\frac{\sgn(t)e^{\rmi\frac{t}{2}}}{2\pi\rmi} \ g\left(\sin^2\left(\frac{t}{2}\right)\right),
\end{split}
\label{t.12}
\eeq
$g(\cdot)$ being the Abel transform of $f(\cdot)$ given in \eqref{1.1}.
\end{proof}

\noindent
It is easy to check from \eqref{t.4} that $\eta(t)$ enjoys the following symmetry:
\beq
\eta(t) = -e^{\rmi t} \eta(-t),
\label{t.14}
\eeq 
which, in view of \eqref{t.3}, induces the following symmetry on the auxiliary Legendre coefficients $\hat{\gamma}_n$:
\beq
\hat{\gamma}_n = -\hat{\gamma}_{-n-1} \qquad (n \in\Z).
\label{t.15}
\eeq

\section{Stability estimates in the presence of noise and regularization}
\label{se:regularization}

In actual problems, there is always some inherent noise affecting the data, e.g., measurement error or (at least)
roundoff numerical error. Therefore, in practice we have to deal with a noisy realization of the data 
$g^{(\varepsilon)} = g + \varepsilon  \, \nu$ (assuming an additive model of noise \cite{Eggermont}), 
where $\nu$ represents a noise function (specified later) and $\varepsilon>0$ is a (small) parameter indicating 
the level of noise. Therefore, instead of expansion \eqref{t.2} we have to handle the following expansion:
\beq
\sum_{n=0}^\infty c_n^{(\varepsilon)} \, \oP_n(x),
\label{r.1}
\eeq
where the noisy Legendre coefficients $c_n^{(\varepsilon)}$ are associated with the noisy data function 
$g^{(\varepsilon)}$ through the auxiliary Fourier coefficients $\hat{\gamma}_n^{(\varepsilon)}$, 
that is (see \eqref{t.3}):
\beq
c_n^{(\varepsilon)} = (-1)^n\,\sqrt{2n+1} \ \hat{\gamma}_n^{(\varepsilon)}
\qquad\qquad (n \geqslant 0),
\label{r.2}
\eeq
with
\beq
\hat{\gamma}_n^{(\varepsilon)}=\int_{-\pi}^\pi \eta^{(\varepsilon)}(t)\,e^{\rmi n t}\,\rmd t 
\qquad\qquad (n \geqslant 0),
\label{r.2.2}
\eeq
and
\beq
\eta^{(\varepsilon)}(t) = \frac{\sgn(t)}{2\pi\rmi}\,e^{\rmi t/2} \, g^{(\varepsilon)}\left(\sin^2\frac{t}{2}\right).
\label{r.3}
\eeq
However, as a manifestation of the ill-posed nature of the inverse Abel transformation \cite{Gorenflo1}, 
expansion \eqref{r.1} does not necessarily converge. Moreover, even if two data functions $g_1$ and $g_2$ do belong
to the range of $A$ and their \emph{distance} in the data space (here $L^2(E)$) is small, nevertheless the
distance between $A^{-1}g_1$ and $A^{-1}g_2$ in the solution space can be arbitrarily large in view of the fact
that the inverse of the compact operator $A$ is not bounded. Therefore we are forced to employ methods 
of regularization. The literature on this topic is very extensive (see, e.g., \cite{DeMicheli1,Engl} 
and the references quoted therein). In this paper we limit ourselves to consider only one of the possible 
approaches to regularization, precisely, the procedure which consists in truncating suitably expansion \eqref{r.1}, 
that is, stopping the summation at certain finite value of $n$ (see, e.g., \cite{DeMicheli2} for a discussion 
of this method).
Therefore, we consider the approximation to the solution $f$ given by:
\beq
f_N^{(\varepsilon)}(x) \doteq \sum_{n=0}^N c_n^{(\varepsilon)} \, \oP_n(x) \qquad\qquad (x\in E),
\label{r.4}
\eeq
where $N=N(\varepsilon;f)$ plays the role of regularization parameter. It is clear that the 
\emph{optimal} truncation index $N$ depends upon the noise level $\varepsilon$ and the smoothness of the 
unknown solution $f$ and, in order to be used in practice, approximation \eqref{r.4} 
must be accompanied by a suitable \emph{a-posteriori} procedure that computes the optimal value of $N$ from the data. 

\subsection{Stability estimates}
\label{subse:stability}

We now consider the problem of the stability of reconstruction \eqref{r.4}. Our goal is 
to estimate how good is approximation \eqref{r.4} in terms of $N$, $\varepsilon$ and smoothness properties 
of $f$. The norm of the error between the (unknown) solution $f$ and the approximant $f_N^{(\varepsilon)}$ 
can be easily bounded by using the triangular inequality:
\beq
\|f-f_N^{(\varepsilon)}\|_{L^2(E)} \leqslant \|f-f_N\|_{L^2(E)} + \|f_N-f_N^{(\varepsilon)}\|_{L^2(E)},
\label{r.5}
\eeq
where
\beq
f_N(x) \doteq \sum_{n=0}^N c_n \, \oP_n(x) \qquad (x\in E),
\label{r.6}
\eeq
represents the approximation to $f$ constructed with noiseless (but unknown) data $g$ (see \eqref{t.2}). 
The first term on the r.h.s. of \eqref{r.5} represents the \emph{approximation error}, 
whereas the second term is the \emph{noise-propagation error}.

Let us introduce on the interval $E$ of the real line the following measure $\mu$:
\beq
\rmd\mu(x)=\frac{\rmd x}{\sqrt{x(1-x)}}	\qquad (0 < x < 1),
\label{r.18.a}
\eeq
where $\rmd x$ is the Lebesgue measure. Let $L^p_\mu(E)$ ($1\leqslant p<\infty$) denote the following
weighted Lebesgue-space:
\beq
L^p_\mu(E) \doteq \left\{f:E\to\R \mathrm{~measurable}
:\,\|f\|_{L^p_\mu(E)}\doteq\left[\int_0^1 |f(x)|^p\,\rmd\mu(x)\right]^\frac{1}{p} \!< \infty\!\right\}.
\label{r.18.L}
\eeq
Smoothness of functions will be measured by the maximum order $k$ of (weak) derivatives with finite norm
in the interval of interest.
Therefore, in accord with \eqref{r.18.L}, we can define for integers $k\geqslant 0$ the following Sobolev space:
\beq
H^k_\mu(E) \doteq \left\{f \in L^2_\mu(E)\,:\,D_x^{j} f(x)
\mathrm{~exists ~ and ~is ~in~} L^2_\mu(E) \mathrm{~for~all~} 0\leqslant j\leqslant k\right\},
\label{r.18.s}
\eeq
where, for convenience, we adopt the notation $D_x^j \equiv d^j/dx^j$. 
This space is equipped with the norm
\beq
\left\|f\right\|_{H^k_\mu(E)} \doteq 
\sum_{j=0}^k \left\|D^{j} f\right\|_{L^2_\mu(E)}. 
\label{r.18.ns}
\eeq
Regarding the \emph{noise-propagation error} we can prove the following proposition.
\begin{proposition}
\label{pro:1}

If the noise on the data is such that:
\beq
\left\|g - g^{(\varepsilon)}\right\|_{L^{2}_\mu(E)} \leqslant\varepsilon \qquad (\varepsilon > 0),
\label{r.6.1}
\eeq
then
\beq
\left\|f_N-f_N^{(\varepsilon)}\right\|_{L^2(E)} \leqslant \frac{(N+1)}{\sqrt{\pi}}\,\varepsilon.
\label{r.6.2}
\eeq
\end{proposition}

\begin{proof}
From \eqref{r.4}, \eqref{r.6} and the orthonormality of the basis $\{\oP(x)\}_{n=0}^\infty$ we have:
\beq
\left\|f_N-f_N^{(\varepsilon)}\right\|_{L^2(E)}^2 = \sum_{n=0}^N \left| c_n-c_n^{(\varepsilon)}\right|^2.
\label{r.7.0}
\eeq
From \eqref{t.11}, \eqref{t.12}, \eqref{r.2} and \eqref{r.3}  we obtain:
\beq
\begin{split}
&\! \left|c_n-c_n^{(\varepsilon)}\right| 
= \frac{\sqrt{2n+1}}{2\pi}\left|\int_{-\pi}^\pi
\left[g\left(\sin^2\frac{t}{2}\right)-g^{(\varepsilon)}\left(\sin^2\frac{t}{2}\right)\right]
\sgn(t)\,e^{\rmi (n+\ud)t}\,\rmd t\right| \\
&\leqslant\!\frac{\sqrt{2n+1}}{\pi}\!\!\int_0^\pi 
\!\left| g\left(\sin^2\frac{t}{2}\right)-g^{(\varepsilon)}\left(\sin^2\frac{t}{2}\right)\right|\rmd t
\!=\!\frac{\sqrt{2n+1}}{\pi}\!\!\int_{0}^1 \!\frac{\left|g(x)-g^{(\varepsilon)}(x)\right|}{\sqrt{x(1-x)}}\rmd x,
\end{split}
\nonumber
\label{r.7}
\eeq
which, by using the Cauchy-Schwarz inequality, yields:
\beq
\left|c_n-c_n^{(\varepsilon)}\right|
\leqslant \sqrt{\frac{2n+1}{\pi}} \ \left\|g-g^{(\varepsilon)}\right\|_{L^2_\mu(E)}.
\label{r.8}
\eeq
From \eqref{r.6.1}, \eqref{r.7.0} and \eqref{r.8} we finally have:
\beq
\left\|f_N-f_N^{(\varepsilon)}\right\|_{L^2(E)} 
\leqslant \frac{\varepsilon}{\sqrt{\pi}}\left(\sum_{n=0}^N (2n+1)\right)^\ud
=\frac{(N+1)}{\sqrt{\pi}}\ \varepsilon.
\label{r.9}
\eeq
\end{proof}

\begin{remark}
By using the H\"older inequality it is easy to show that:
$\|g\|_{L^2_\mu(E)} \leqslant A_p \, \|g\|_{L^p(E)}$ for all $p>4$,
$A_p$ being a constant depending only on $p$.
In particular, if $\|g - g^{(\varepsilon)}\|_{L^\infty(E)} \leqslant\varepsilon$, i.e., in the case of 
uniformly bounded noise, from \eqref{r.7.0} and \eqref{r.8} it follows:
$\|f_N-f_N^{(\varepsilon)}\|_{L^2(E)}\leqslant (N+1)\,\varepsilon$.
\end{remark}

Concerning the \emph{approximation error} we have:
\beq
\|f-f_N\|_{L^2(E)}^2 = \left\|\sum_{n=N+1}^\infty c_n \oP_n(x)\right\|_{L^2(E)}^2
=\sum_{n=N+1}^\infty |c_n|^2=\sum_{n=N+1}^\infty (2n+1)|\hat{\gamma}_n|^2,
\label{r.10}
\eeq
$\hat{\gamma}_n$ being the Fourier coefficients of the $2\pi$-periodic function 
$\eta(t)$ ($t\in\R$) (see \eqref{t.3}).
It is well-known that the asymptotic behavior, for large $n$, of the coefficients $\hat{\gamma}_n$ is
related to the smoothness of the periodic function $\eta(t)$ which, in turn, here depends on the smoothness 
of the data function $g(t)$ (see \eqref{t.4}).
In this context, it seems therefore natural to associate the \emph{approximation error} \eqref{r.10} 
with functions defined on the circle $S^1$ which belong to the Sobolev space
\beq
H^k(S^1) \doteq \left\{u \in L^2(S^1) \,:\, 
\|u\|_{H^k(S^1)} \doteq \left[\sum_{n\in\Z} (1+n^2)^k\,|\hat{u}_n|^2\right]^\ud \!< \infty \right\} 
\quad \! (k \geqslant 0),
\label{r.10.a}
\eeq
where $\hat{u}_n$ ($n\in\Z$) denotes the Fourier coefficients of the function $u$. 
In parallel, we can define an equivalent norm in $H^k(S^1)$ in terms of derivatives 
of $u$, i.e.:
\beq
\|u\|'_{H^k(S^1)} \doteq \left(\sum_{j=0}^k \left\|D^j u\right\|_{L^2(\Omega)}^2\right)^\ud.
\label{r.10.b}
\eeq
The equivalence between the norms $\|u\|_{H^k(S^1)}$ and $\|u\|'_{H^k(S^1)}$ 
means that there exist two constants, $q_1(k)$ and $q_2(k)$, such that
\beq
q_1\,\|u\|'_{H^k(S^1)} \leqslant  \|u\|_{H^k(S^1)} \leqslant q_2\,\|u\|'_{H^k(S^1)}.
\label{r.10.c}
\eeq
In particular, from definitions \eqref{r.10.a} and \eqref{r.10.b} we have: 
$q_2(k) = 2\pi\binom{k}{\lfloor k/2\rfloor}$.
Now, if $\eta(t) \in H^k(S^1)$, then its Fourier coefficients $\hat{\gamma}_n$ satisfy 
the following bound, which follows immediately from \eqref{r.10.a}:
$|\hat{\gamma}_n| \leqslant \frac{\|\eta\|_{H^k(S^1)}}{(1+n^2)^{k/2}}$.
Therefore, from \eqref{r.10} we have:
\beq
\|f-f_N\|_{L^2(E)}^2
\leqslant \|\eta\|_{H^k(S^1)}^2\sum_{n=N+1}^\infty \frac{2n+1}{(1+n^2)^k}
\leqslant 3\,\|\eta\|_{H^k(S^1)}^2\sum_{n=N+1}^\infty \frac{1}{n^{2k-1}},
\label{r.13}
\eeq
the latter sum being convergent for $k>1$. Now, from the estimate for the \emph{tail} of a convergent series 
by the integral test we have:
$\sum_{n=(N+1)}^\infty n^{1-2k}\leqslant\int_N^\infty x^{1-2k}\rmd x = [2(k-1) N^{2k-2}]^{-1}$, 
which finally yields:
\beq
\|f-f_N\|_{L^2(E)}\leqslant q_3(k) \, \frac{\|\eta\|_{H^k(S^1)}}{N^{k-1}}
\qquad \left(\eta(t)\in H^k(S^1); k > 1\right),
\label{r.14}
\eeq
where $q_3(k)=\sqrt{\frac{3}{2(k-1)}}$.
Now, we have to relate the smoothness properties of the $2\pi$-periodic function $\eta(t)$ ($t\in\R$) 
to those of the data function $g(x)$ ($x\in E$).
We begin by proving the following lemma.
\begin{lemma}
\label{lem:2}

If $g\in H^k_\mu(E)$ and satisfies the following relations:
\beq
\lim_{t\to 0} \, \left[D_t^{j} g\left(\txs\sin^2\frac{t}{2}\right)\right] = 0 \qquad (j=0,1,\ldots,k-1),
\label{rr.1}
\eeq
then:
\begin{itemize} 
\item[(i)] the $j$-th (weak) derivative $D_t^{j}\eta(t)$ of $\eta(t)$ is the 
$2\pi$-periodization of the function defined on $\Omega$ by:
\beq
D_t^{j}\eta(t) = \frac{\sgn(t)}{2\pi\rmi} \
D_t^{j}\left[e^{\rmi t/2}g\left(\sin^2\txs\frac{t}{2}\right)\right] 
\qquad (j = 0, 1, \ldots,k; \, t\in\Omega);
\label{r.15}
\eeq
\item[(ii)] $\eta\in H^k(S^1)$;
$\left\|\eta\right\|'_{H^k(S^1)} \leqslant q_4(k) \, \left\|g\right\|_{H^k_\mu(E)}$
with $q_4(k)=\sqrt{2(k+1)}k!^2/\pi$.
\end{itemize}
\end{lemma}

\begin{proof}
Let $h(t) \doteq e^{\rmi t/2}g(\sin^2(t/2))/(2\pi\rmi)$ so that $\eta(t)=\sgn(t)\,h(t)$. 
From the definition of weak derivative \cite{Adams},
for any test function $\phi\in C^\infty_0(\Omega)$ successive integrations by parts give:
\beq
\begin{split}
&\int_{-\pi}^\pi\eta(t) \, D_t^{k}\phi(t)\,\rmd t 
=-\int_{-\pi}^0 h(t)\,D_t^{k}\phi(t)\,\rmd t+\int_0^{\pi} h(t) \, D_t^{k}\phi(t)\,\rmd t \\
&\quad=-\sum_{j=0}^{k-1}\!\left[(-1)^j \, D_t^{j} h(t) \ D_t^{(k-j-1)}\phi(t)\right]_{t=-\pi}^0
\!+\!(-1)^{k+1}\!\!\int_{-\pi}^0 \!\! D_t^{k}h \, \phi(t)\,\rmd t \\
&\qquad\,+\sum_{j=0}^{k-1}\!\left[(-1)^j\,D_t^{j} h(t) \, D_t^{(k-j-1)} \phi(t)\right]_{t=0}^{\pi}
+(-1)^k\!\int_0^{\pi}\,D_t^{k}h(t) \, \phi(t)\,\rmd t \\
&\quad=-2\sum_{j=0}^{k-1}(-1)^j \left[D_t^{j}h(t) \, D_t^{(k-j-1)}\phi(t)\right]_{t=0}
\!+(-1)^k\!\!\int_{-\pi}^{\pi}\!\left[\sgn(t)\,D^k_t h(t)\right]\phi(t)\,\rmd t,
\end{split}
\label{r.15.a}
\eeq
where, for $j=0,1,\ldots,k-1$, the periodicity of the test functions: 
$D_t^{j} \phi(-\pi)=D_t^{j}\phi(\pi)$ has been used together with the relation: 
$D_t^{j} h(-\pi)+D_t^{j} h(\pi)=0$, which follows immediately from
the definition of $h(t)$. If $[D_t^{j}\,h(t)]_{t=0}=0$  ($j=0,1,\ldots,k-1$), then from
\eqref{r.15.a} it follows that 
\beq
\int_{-\pi}^\pi\eta(t) \, D_t^{k}\phi(t)\,\rmd t 
=(-1)^k\int_{-\pi}^{\pi}\left[\sgn(t)\,D^k_t h(t)\right]\phi(t)\,\rmd t,
\label{r.15.b}
\eeq
that is, $\left[\sgn(t)\,D^k_t h(t)\right]$ is the $k$-th weak derivative of $\eta(t)$.
It is finally easy to see that the conditions $[D_t^{j}\,h(t)]_{t=0}=0$ ($j=0,1,\ldots,k-1$) are satisfied if 
conditions \eqref{rr.1} hold true. Point (i) is thus proved. \\
For what regards the point (ii), our goal now is to find an upper bound for $\left\|\eta\right\|'_{H^k(S^1)}$ 
(see \eqref{r.10.b}) in terms of the Sobolev norm of $g$. Using \eqref{r.15}, the product rule 
for the derivatives, and the Minkowski inequality, we have:
\beq
\begin{split}
\left\|D^{k} \eta\right\|_{L^2(\Omega)}
&=\left[\int_{-\pi}^\pi\left|D^{k}_t \eta(t)\right|^2\,\rmd t\right]^\ud 
=\left[\int_{-\pi}^\pi\left|\sgn(t)\,D^{k}_t h(t)\right|^2\,\rmd t\right]^\ud \\
&=\frac{1}{2\pi}\left[\int_{-\pi}^\pi
\left|e^{\rmi t/2}\sum_{j=0}^k\binom{k}{j}\left(\frac{\rmi}{2}\right)^{(k-j)}
D^{j}_t g\left(\txs\sin^2\frac{t}{2}\right)\right|^2\rmd t\right]^\ud \\
&\leqslant\sum_{j=0}^k d_{k,j}\left[\int_{-\pi}^\pi
\left|D^{j}_t g\left(\txs\sin^2\frac{t}{2}\right)\right|^2\rmd t\right]^\ud,
\end{split}
\label{r.18}
\eeq
with $d_{k,j}=2^{(j-k-1)}\binom{k}{j}/\pi$.
Let $x(t)=\sin^2(t/2)$; then: $2D_t = \sin t \,D_x$. Now, the Fa\`a di Bruno formula
for the derivatives of the composition $g(x(t))$ reads:
\beq
D_t^j\,g(x(t)) = \sum_{\ell=0}^j D_x^{\ell} \, g(x(t)) \cdot
B_{j,\ell}\left(D_t x(t), D_t^2 x(t), \ldots, D_t^{(j-\ell+1)} x(t)\right),
\label{r.18.n}
\eeq
where $B_{j,\ell}(y_1,y_2,\ldots,y_{j-\ell+1})$ are the (partial) Bell polynomials \cite{Comtet}: 
\beq
B_{j,\ell}\left(y_1,y_2,\ldots,y_{j-\ell+1}\right)
=\sum\frac{j!}{\ell_1!\ell_2!\cdots \ell_{j}!}\,\prod_{i=1}^j\left(\frac{y_i}{i!}\right)^{\ell_i},
\label{r.18.nn}
\eeq
the summation being over all partitions of $j$ into $\ell$ non-negative parts, i.e., over all non-negative
integer solutions $\ell_1,\ell_2,\ldots,\ell_{j}$ of the two equations
\beq
\ell_1+2\ell_2+\ldots+j\ell_{j} = j, \qquad
\ell_1+\ell_2+\ldots+\ell_{k} = \ell.
\label{rr}
\eeq
Note that: $B_{0,0}=1$, $B_{0,\ell}=0$ for $\ell\geqslant 1$.
Now, from \eqref{r.18.n} we have:
\beq
\begin{split}
&\left[\int_{-\pi}^\pi \left|D_t^{j}\, g\left(\txs\sin^2\frac{t}{2}\right)\right|^2 \rmd t\right]^\ud \\
&=\left[\int_{-\pi}^\pi \left|\sum_{\ell=0}^j D_x^{\ell} \, g(x(t)) \cdot 
B_{j,\ell}\left(D_t x(t), D_t^2 x(t), \ldots, D_t^{(j-\ell+1)} x(t)\right)\right|^2\rmd t\right]^\ud \\
&=\left[2\!\!\int_{0}^1 \left|\sum_{\ell=0}^j D_x^{\ell}\, g(x) \!\cdot\! 
\left[B_{j,\ell}\left(D_t x(t),D_t^2 x(t),\ldots,D_t^{(j-\ell+1)} x(t)\right)\right]_{t=t_*}
\right|^2\!\!\rmd\mu(x)\right]^\ud \\
&\leqslant\sum_{\ell=0}^j
\left[2\!\int_{0}^1 \left|D_x^{\ell}\, g(x)\right|^2 \!\cdot\! 
\left|B_{j,\ell}\left(D_t x(t),D_t^2 x(t),\ldots,D_t^{(j-\ell+1)}x(t)\right)
\right|^2_{t=t_*}\!\rmd\mu(x)\right]^\ud\!\!\!,
\end{split}
\label{r.18.b}
\eeq
where $t_*(x)=\cos^{-1}(1-2x)$.
Since $|D_t^{n}\, x(t)|_{t=t_*(x)}=\ud|\cos(t+n\frac{\pi}{2})|_{t=t_*(x)} \leqslant\ud$ 
for $x\in[0,1]$ and $n=1,2,\ldots$, we have for $x\in[0,1], j=0,1,\ldots, \mathrm{\,and\,} 
\,0\leqslant\ell\leqslant j$:
\beq
\begin{split}
&\left|B_{j,\ell}\left(D_t x(t),D_t^2 x(t),\ldots,D_t^{(j-\ell+1)}x(t)\right)
\right|_{t=t_*(x)}\leqslant B_{j,\ell}\txs\left(\ud,\ud,\ldots,\ud\right) \\
&\quad\leqslant\frac{1}{2^\ell}\sum\frac{j!}{\ell_1!\ell_2!\cdots\ell_{j}!}
=\frac{j!}{2^\ell\,\ell!}\binom{j-1}{\ell-1},
\end{split}
\label{r.18.c}
\eeq
where we used $\sum\frac{\ell!}{\ell_1!\ell_2!\cdots\ell_{j}!}=\binom{j-1}{\ell-1}$ under conditions \eqref{rr}.
From \eqref{r.18.b} and \eqref{r.18.c} we thus obtain:
\beq
\left[\!\int_{-\pi}^\pi\!\left|D_t^{j}\, g\left(\txs\sin^2\frac{t}{2}\right)\right|^2 \!\rmd t\right]^\ud
\!\!\leqslant\!\sum_{\ell=0}^j \! b_{j,\ell}
\left[\int_{0}^1 \! \left|D_x^{\ell} \, g(x)\right|^2 \rmd\mu(x)\right]^\ud
\!\!=\sum_{\ell=0}^jb_{j,\ell}\left\|D^{\ell} g\right\|_{L^2_\mu(E)},
\label{r.18.d} \nonumber
\eeq
where $b_{j,\ell}=\frac{j!}{2^{(\ell-1/2)}\,\ell!}\binom{j-1}{\ell-1}$.
The latter inequality and Eq. \eqref{r.18} then give:
\beq
\begin{split}
\left\|D^k\eta\right\|_{L^2(\Omega)} 
&\leqslant\sum_{j=0}^k\sum_{\ell=0}^j d_{k,j}\,b_{j,\ell} \left\|D^\ell g\right\|_{L^2_\mu(E)}
=\sum_{\ell=0}^k\left(\sum_{j=\ell}^k d_{k,j}\,b_{j,\ell}\right) \left\|D^\ell g\right\|_{L^2_\mu(E)} \\
&\quad\leqslant q_k \sum_{\ell=0}^k \left\|D^\ell g\right\|_{L^2_\mu(E)}
=q_k \left\|g\right\|_{H^k_\mu(E)},
\end{split}
\label{r.19}
\eeq
where $q_k=\sqrt{2}k!^2/\pi$. Finally, from \eqref{r.10.b} and \eqref{r.19} we obtain:
\beq
\left\|\eta\right\|'_{H^k(S^1)} \doteq \left(\sum_{j=0}^k \left\|D^j\eta\right\|^2_{L^2(\Omega)}\right)^\ud
\leqslant q_k\left(\sum_{j=0}^k \left\|g\right\|^2_{H^j_\mu(E)}\right)^\ud
\leqslant q_4(k) \, \left\|g\right\|_{H^k_\mu(E)},
\label{r.20}
\eeq
where $q_4(k)=\sqrt{k+1}\,q_k$, and the last inequality following from the embedding relation 
among the Sobolev spaces $H_\mu^j(E)$ for $j=0,\ldots,k$ \cite{Adams}.
\end{proof}
We can now state quantitatively that the smoother the inverse Abel transform 
$f(x)$ is, the faster its approximation $f_N(x)$ converges to $f(x)$.

\begin{proposition}
\label{pro:2}

If $g\in H_\mu^k(E)$ ($k > 1$) and conditions \eqref{rr.1} are satisfied, then the approximation error
can be bounded as follows:
\beq
\|f-f_N\|_{L^2(E)}\leqslant q_5(k)\,\frac{\|g\|_{H^k_\mu(E)}}{N^{k-1}} \qquad (k>1),
\label{r.21}
\eeq
with $q_5(k)=q_2(k)q_3(k)q_4(k)$.
\end{proposition}

\begin{proof}
Bound \eqref{r.21} follows plugging \eqref{r.20} into \eqref{r.10.c} and then the result into \eqref{r.14}.
\end{proof}

Finally, we can prove in the next theorem how the truncation index $N$ should depend on $\varepsilon$ and on 
the smoothness of the solution $f$ in order to have guaranteed a reconstruction error of the order of the noise 
level.
\begin{theorem}
\label{the:1}

Assume the noise on the data $g$ is such that:
$\|g - g^{(\varepsilon)}\|_{L^{2}_\mu(E)}\leqslant\varepsilon$, $\varepsilon > 0$, and let
$g \in H^k_\mu(E)$. If the truncation index is $N=c\,\varepsilon^{-1/k}$ ($c$=constant) then we have for $k>1$:
\beq
\left\|f-f_N^{(\varepsilon)}\right\|_{L^2(E)} = O\left(\varepsilon^{(k-1)/k}\right) \qquad (k>1).
\label{r.22}
\eeq
\end{theorem}

\begin{proof}
Bound \eqref{r.22} follows from \eqref{r.5} and bounds \eqref{r.6.2} and \eqref{r.21}
given in Propositions \ref{pro:1} and \ref{pro:2}, respectively.
\end{proof}

\section{Numerical analysis: Algorithm and examples}
\label{se:numerical}

Formula \eqref{r.4} provides us with a one parameter set of regularized solutions 
$\{f_N^{(\varepsilon)}(x)\}_{N\geqslant 0}$, the integer $N$ playing the role of regularization parameter,
from which we have to select a solution which can be considered in some sense an optimal solution to problem
\eqref{1.1}. The indication given by Theorem \ref{the:1} on the optimal value of the parameter $N$ is important
from the theoretical point of view but it is of little
practical utility since it requires the \emph{a-priori} knowledge of both noise level and smoothness properties
of the solution. Although it is often reasonable to have estimates of the amount of noise corrupting the data,
instead it is usually impossible to check if \emph{a-priori} smoothness assumptions on the 
solution are actually satisfied. It is therefore appropriate to adopt \emph{a-posteriori} strategies,
which can determine a suitable regularization parameter $N$ from the data, without making any assumption 
on the smoothness of the solution. In this context Morozov's Discrepancy Principle seems to be an
appropriate method. Accordingly, the parameter $N$ is chosen in such a way that 
\beq
\left\|A\,f_N^{(\varepsilon)}-g^{(\varepsilon)}\right\|_{L^2(E)} = \tau\,\varepsilon
\qquad (\tau>1),
\label{n.1}
\eeq
which amounts to requiring that \emph{given data} and \emph{reproduced data} coincide within the noise
(for a proof guaranteeing the existence of a solution to \eqref{n.1}, see \cite{Groetsch}).

\skd

The algorithm for computing the solution to problem \eqref{1.1} which emerges from this analysis is very simple and
can be summarized in the following three steps:

\begin{enumerate}
\item[\textbf{1.}] From the data $g^{(\varepsilon)}$ 
compute the Fourier coefficients $\gamma_n^{(\varepsilon)}$ of the function 
$\eta^{(\varepsilon)}(t)$ (see \eqref{r.2.2} and \eqref{r.3}), and then calculate the coefficients 
$c_n^{(\varepsilon)}$ by formula \eqref{r.2}.
\item[\textbf{2.}] According to criterion \eqref{n.1}, fix $\tau>1$ and choose the truncation index $N$ such that
\beq
\left\|A\,f_N^{(\varepsilon)}-g^{(\varepsilon)}\right\|_{L^2(E)} \leqslant \tau\,\varepsilon
<\left\|A\,f_{N+1}^{(\varepsilon)}-g^{(\varepsilon)}\right\|_{L^2(E)}. \nonumber
\label{n.2}
\eeq
\item[\textbf{3.}] Compute the solution $f_N^{(\varepsilon)}(x)$ by formula \eqref{r.4}.
\end{enumerate}

\vskip 0.2 cm

The first step is rapid since it can take full advantage of the computational efficiency of the 
Fast Fourier Transform both in terms of speed of computation and of accuracy \cite{Calvetti}.
The first $N$ Fourier coefficients 
$\gamma_n^{(\varepsilon)}$ can be computed by a single FFT from the samples of $\eta^{(\varepsilon)}(t)$ 
at $N$ distinct points of the interval $[-\pi,\pi]$ in $\mathcal{O}(N \log N)$ time. 
Even the third step can be implemented  efficiently by using fast algorithms for the evaluation 
of Legendre expansions (see, e.g., the $\mathcal{O}(N \log N)$ algorithm proposed in \cite{Alpert}).
The criterion in Step 2 guarantees high accuracy at the expense of computing time performance.
More efficient tests (though, in general, less accurate) which are based on \emph{a-priori} decision
criteria, can alternately be implemented (see, e.g., \cite{DeMicheli1,DeMicheli2}).

\subsection{Numerical examples}
\label{subse:numerical}

The above algorithm for the computation of the inverse Abel transform has been implemented in double precision 
arithmetic using the standard routines of the open source GNU Scientific Library
for the computation of the Fast Fourier Transform and of the Legendre polynomials. 
The performances of the algorithm have been evaluated on numerous test functions, 
some of them (or slight variants of them) have been considered also by other authors
in the literature (see, e.g., the numerical experiments in Refs. 
\cite{Ammari1,Ammari2,Deutsch,Gueron,Minerbo,Pandey}).

\begin{figure}[tb]
\includegraphics[width=\textwidth]{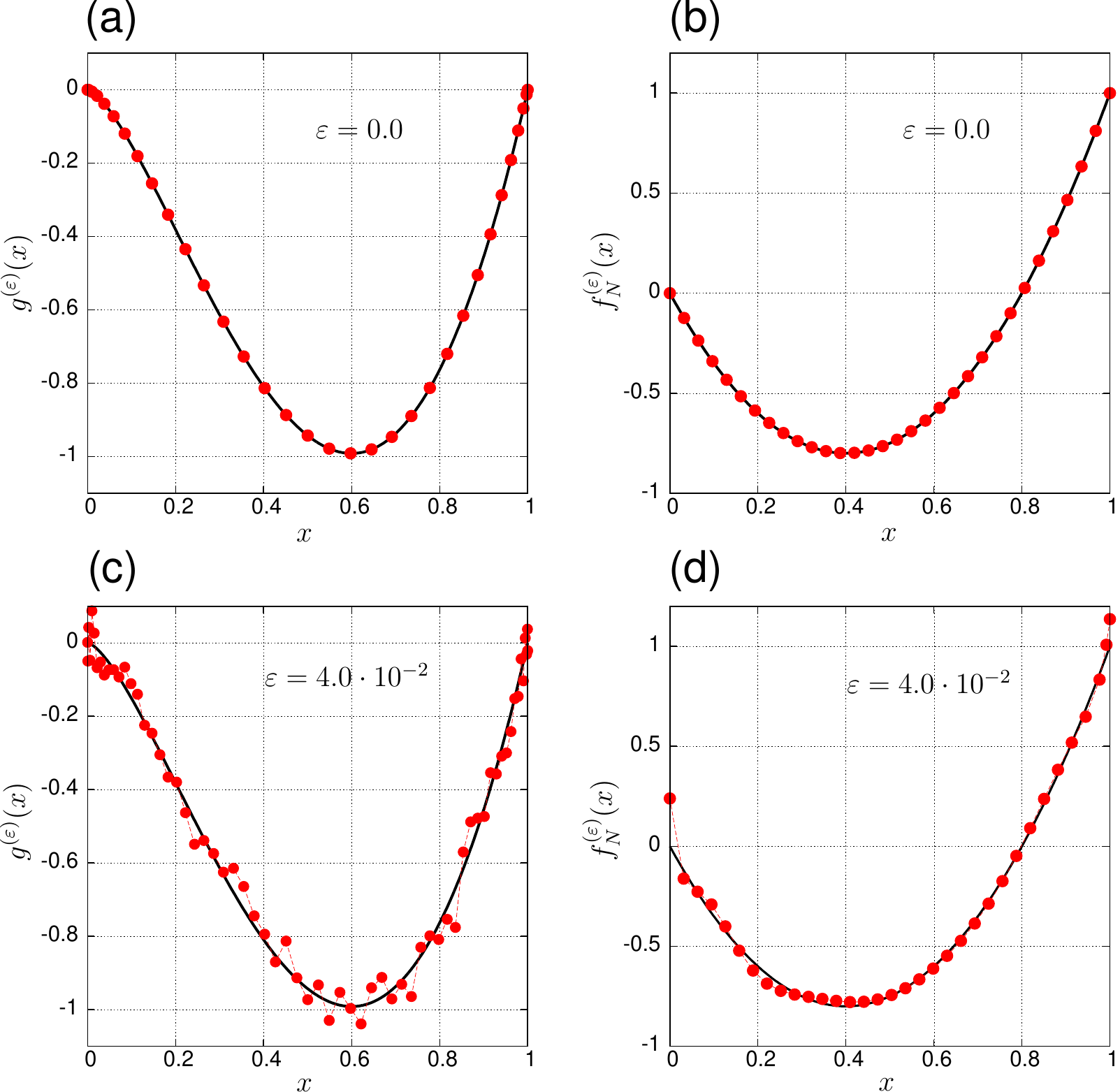} 
\caption{\label{fig:1} \small 
Analysis of the function $f_1(x)$ (see \eqref{ammari0.f}). 
(a): Abel transform $g_1(x)$ (solid line; see Eq. \eqref{ammari0.g}). 
The filled dots represent the input data samples $g_1(\sin^2 (x_j/2))$ 
(see \eqref{t.4} and \eqref{r.3}), $\{x_j\}_{j=1}^{N_s}$ being $N_s$ points of 
a uniform grid on the interval $[0,1]$; $N_s = 64$; 
$\varepsilon = 0.0$, i.e., the input data samples are affected only by roundoff error.
(b): Inverse Abel transform $f_1(x)$ (solid line) of the function $g_1(x)$. 
The dots represent samples of the approximation $f_N^{(\varepsilon)}(x)$ 
(see \eqref{r.4}) computed from the input data shown in (a) with $N=13$, 
which corresponds to the minimum discrepancy \eqref{n.1} (see next Fig. \ref{fig:2}(c)). 
(c): Noisy input data $g_1^{(\varepsilon)}(\sin^2 (x_j/2))$.
The noise level is: $\varepsilon = 4.0\cdot10^{-2}$; 
the corresponding signal-to-noise ratio of the input data is $\mathrm{SNR}=23.1\,\mathrm{dB}$.
(d) Inverse Abel transform $f_N^{(\varepsilon)}(x)$ computed from the data shown in (c) by Eq. \eqref{r.4} with $N = 13$.}
\end{figure}

Figure \ref{fig:1} summarizes the results of recovering the smooth function $f_1(t)$
from its Abel transform $g_1(x)$, where \cite{Ammari1,Ammari2}:
\begin{subequations}
\begin{align}
f_1(x) &= 5x^2-4x, \label{ammari0.f} \\
g_1(x) &= (Af_1)(x) = \frac{16}{3}(x^{5/2}-x^{3/2}). \label{ammari0.g}
\end{align}
\end{subequations}
In panel (b) we see the excellent reconstruction of $f_1(x)$ (filled dots) obtained for the noise level 
$\varepsilon = 0.0$, which amounts to saying that the input data samples $\{g_1(x_j)\}_{j=1}^{N_s}$
were affected by only roundoff error (see the figure legend for numerical details). 
Note that the input samples $g_1(x_j)$ (shown in Fig. \ref{fig:1}(a)) are not equispaced 
since the actual function to be Fourier transformed, and therefore to be sampled on a regular grid, 
is the function $\eta(t)$ on $ t\in[-\pi,\pi)$ (see \eqref{t.4}), which is proportional to $g(\sin^2(\frac{t}{2}))$.
If the actual data are not available on the prescribed grid on the interval $[-\pi,\pi)$, then fast routines
computing the Fourier transform at nonequispaced points can be conveniently used 
(see, e.g., \cite{Greengard}).

Panels (c) and (d) show the analysis in the case of noisy input data (see Fig. \ref{fig:1}(c)).
Noisy data samples have been obtained by adding to the noiseless data $\{g_1(x_j)\}$ random numbers normally 
distributed with variance $\varepsilon^2$. In this example we had $\varepsilon=4.0\cdot10^{-2}$, which corresponds
to a signal-to-noise ratio of the input data (defined as the ratio of the noiseless data power to
the noise power) equal to $\mathrm{SNR} = 23.1\,\mathrm{dB}$. In spite of the rather low $\mathrm{SNR}$,
the accurate reconstruction shown in Fig. \ref{fig:1}(d) exhibits the good stability of the inversion algorithm.
As expected from the theoretical analysis where, in order to estimate the propagation error, 
we have been forced to introduce weighted Lebesgue (and Sobolev) spaces with measure $d\mu$ (see \eqref{r.18.a}),
larger errors appear close to the boundary points, i.e., for $x\sim0$ and $x\sim 1$, nonetheless remaining
rather limited. This behavior is confirmed and made more clear in Fig. \ref{fig:2}(a), where the pointwise error 
$\left|f^{(\varepsilon)}(x)-f^{(\varepsilon)}_N(x)\right|$ is plotted against $x\in[0,1]$ for three
levels of noise. Figure \ref{fig:2}(b) exhibits the (exponentially) rapid descent of the root mean square error
committed in computing the Abel inversion as the SNR of the input data increases 
(or, equivalently, as $\varepsilon$ decreases), showing the stability of the algorithm 
against the noise, argued earlier from the analysis of Fig. \ref{fig:1}. 
It is however worth recalling that in this example $f_1(x)$ is an analytic functions and, 
therefore, its Legendre expansion is expected to converge spectrally fast. This contributes indeed to the very high stability manifested in this case.

\begin{figure}[tb]
\includegraphics[width=\textwidth]{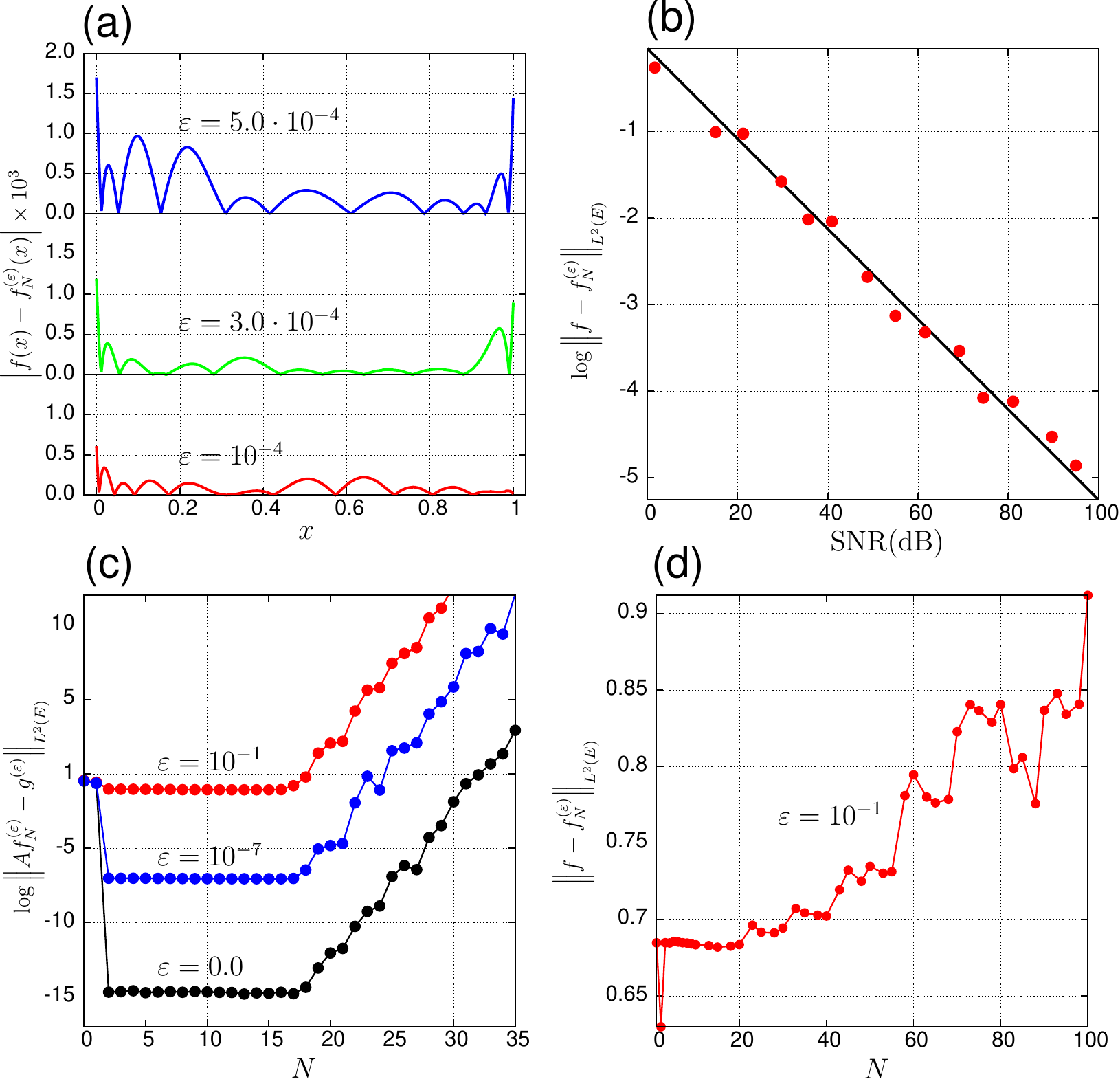} 
\caption{\label{fig:2} \small Function $f_1(x)$ (see \eqref{ammari0.f}): Error analysis.
(a) Pointwise error $|f(x)-f_N^{(\varepsilon)}(x)|$ vs. $x$ (in logarithmic scale) computed 
for three levels of noise $\varepsilon$ on the data $g_1^{(\varepsilon)}(x)$: 
$\varepsilon_1=10^{-4}$ ($\mathrm{SNR}=75\,\mathrm{dB}$),
$\varepsilon_2=3.0\cdot10^{-4}$ ($\mathrm{SNR}=65\,\mathrm{dB}$), 
$\varepsilon_3=5.0\cdot10^{-4}$ ($\mathrm{SNR}=60\,\mathrm{dB}$). The corresponding 
(minimum discrepancy) truncation indexes were: $N_1=15, N_2=14, N_3=14$.
(b) Root mean square error $\|f-f_N^{(\varepsilon)}\|_{L^2(E)}$ 
(in logarithmic scale) against $\mathrm{SNR}$.
The approximating solid line represents: 
$\exp(ax+b)$ with $a=-0.12, b=-0.1$. (c) Examples of discrepancy 
$\|Af_N^{(\varepsilon)}-g^{(\varepsilon)}\|_{L^2(E)}$ against $N$ for various levels of noise: 
$\varepsilon=0.0$, $\varepsilon=10^{-7}$, $\varepsilon=10^{-1}$.
(d) Root mean square error $\|f-f_N^{(\varepsilon)}\|_{L^2(E)}$ against $N$ with 
$\varepsilon=10^{-1}$ ($\mathrm{SNR}=15.2\,\mathrm{dB}$). In this case the approximants $f_N^{(\varepsilon)}(x)$ have been computed by using $N_s=128$ input data samples.}
\end{figure}

Panels (c) and (d) of Fig. \ref{fig:2} give an example of the discrepancy analysis which aims at selecting the 
\emph{optimal} truncation index $N$. The discrepancy function (shown in Fig. \ref{fig:2}(c) for three levels
of noise) exhibits a \emph{plateau} ranging nearly from $N\gtrsim 3$ through $N\lesssim 17$ before starting 
diverging very rapidly for $N > 17$. For values of $N$ belonging to this \emph{plateau}, the Abel inversion 
algorithm is expected to yield a \emph{nearly-optimal} reconstruction of the function $f_1(x)$. 
This is indeed the case, as shown in Fig. \ref{fig:2}(d) where the $L^2$-reconstruction-error on the 
function $f_1(x)$ is plotted against $N$. Parallel to what shown in panel (c), the $L^2$-error remains small 
for $N\lesssim 20$ before starting to increase for larger values of $N$.

\begin{figure}[tb]
\includegraphics[width=\textwidth]{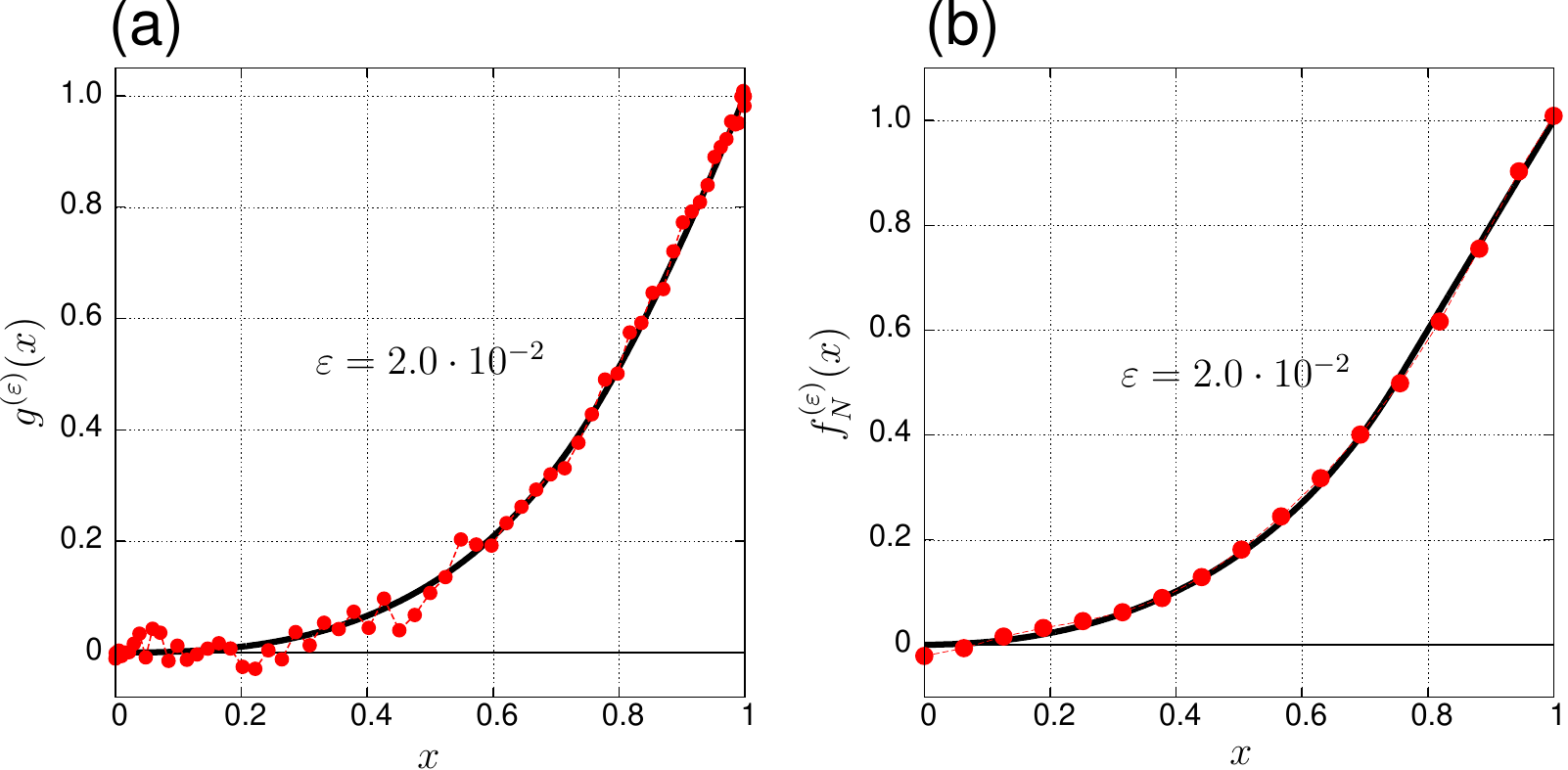} 
\caption{\label{fig:3} \small Analysis of the function $f_2(x)$ (see \eqref{Deutsch.f}).
(a): Abel transform $g_2(x)$ (solid line; see \eqref{Deutsch.g}). 
The filled dots represent the noisy input data samples $g_2^{(\varepsilon)}(\sin^2 (x_j/2))$ (see \eqref{r.3}); $N_s = 64$; $\varepsilon = 2.0\cdot10^{-2}$;
$\mathrm{SNR}=27.5\,\mathrm{dB}$.
(b): Inverse Abel transform $f_2(x)$ (solid line) of the function $g_2(x)$. 
The filled dots represent samples of the approximation $f_N^{(\varepsilon)}(x)$ (see \eqref{r.4}) 
computed from the input data shown in (a) with truncation index $N=15$, which corresponds to the minimum discrepancy \eqref{n.1}.
}
\end{figure}

Figure \ref{fig:3} shows the analysis of the Abel-pair of functions \cite{Deutsch,Gueron}:
\begin{subequations}
\begin{align}
f_2(x) &= 2I_{[0,\frac{2}{3}]}(x)(2-2\sqrt{1-x}-x)+I_{[\frac{2}{3},1]}(x)(2x-1) \label{Deutsch.f}, \\[+3pt]
g_2(x) &=
\begin{cases}
\frac{4}{3}(3-2x)\sqrt{x}-4(1-x)\log\frac{1+\sqrt{x}}{\sqrt{1-x}} & \mathrm{if}~x \in [0,\frac{2}{3}), \\[+5pt]
\frac{4}{3}(3-2x)\sqrt{x}-\frac{9-8x}{3}\sqrt{4x-3}-4(1-x)\log\frac{2+2\sqrt{x}}{1+\sqrt{4x-3}}
& \mathrm{if}~x \in [\frac{2}{3},1],
\end{cases}
\label{Deutsch.g}
\end{align}
\end{subequations}
where $\rmI_{[a,b]}(x)$ denotes the indicator function of the interval $[a,b]$.
Figure \ref{fig:3}(b) shows the excellent recovery of the inverse Abel transform $f_2(x)$, computed
from the noisy data displayed in Fig. \ref{fig:3}(a) ($\mathrm{SNR}=27.5\,\mathrm{dB}$), 
even in the present case where the function to be recovered is not analytic.

In the last example we consider the inverse Abel integral problem \eqref{1.1} with discontinuous solution.
The test Abel-pair is \cite{Ammari2}:

\begin{subequations}
\begin{align}
&f_3(t)=1-\rmI_{[0,\frac{1}{5}]}(t)+(1-t)\rmI_{[\frac{1}{5},\ud]}(t)
+(t-1)\rmI_{[\ud,\frac{7}{10}]}(t)+\ud\rmI_{[\frac{7}{10},1]}(t) \label{Ammari2.f} \\
&g_3(x)= 
\begin{cases}
0 & \!\!\!\!\mathrm{if}~x \in [0,\frac{1}{5}), \\[+3pt]
\frac{2}{3}(x-0.2)^\frac{3}{2}-(2x-4)(x-0.2)^\ud \equiv g_3^{(a)}(x) 
& \!\!\!\!\mathrm{if}~x \in [\frac{1}{5},\ud), \\[+3pt]
g_3^{(a)}(x)-\frac{4}{3}(x-0.5)^\frac{3}{2}+(4x-4)(x-0.5)^\ud \equiv g_3^{(b)}(x) & 
\!\!\!\!\mathrm{if}~x \in [\ud,\frac{7}{10}), \\[+3pt]
g_3^{(b)}(x)+\frac{2}{3}(x-0.7)^\frac{3}{2}-(2x-3)(x-0.7)^\ud & \!\!\!\!\mathrm{if}~x \in [\frac{7}{10},1].
\end{cases}
\label{Ammari2.g}
\end{align}
\end{subequations}
and the results are shown in Fig. \ref{fig:4}. 

\begin{figure}[tb]
\includegraphics[width=\textwidth]{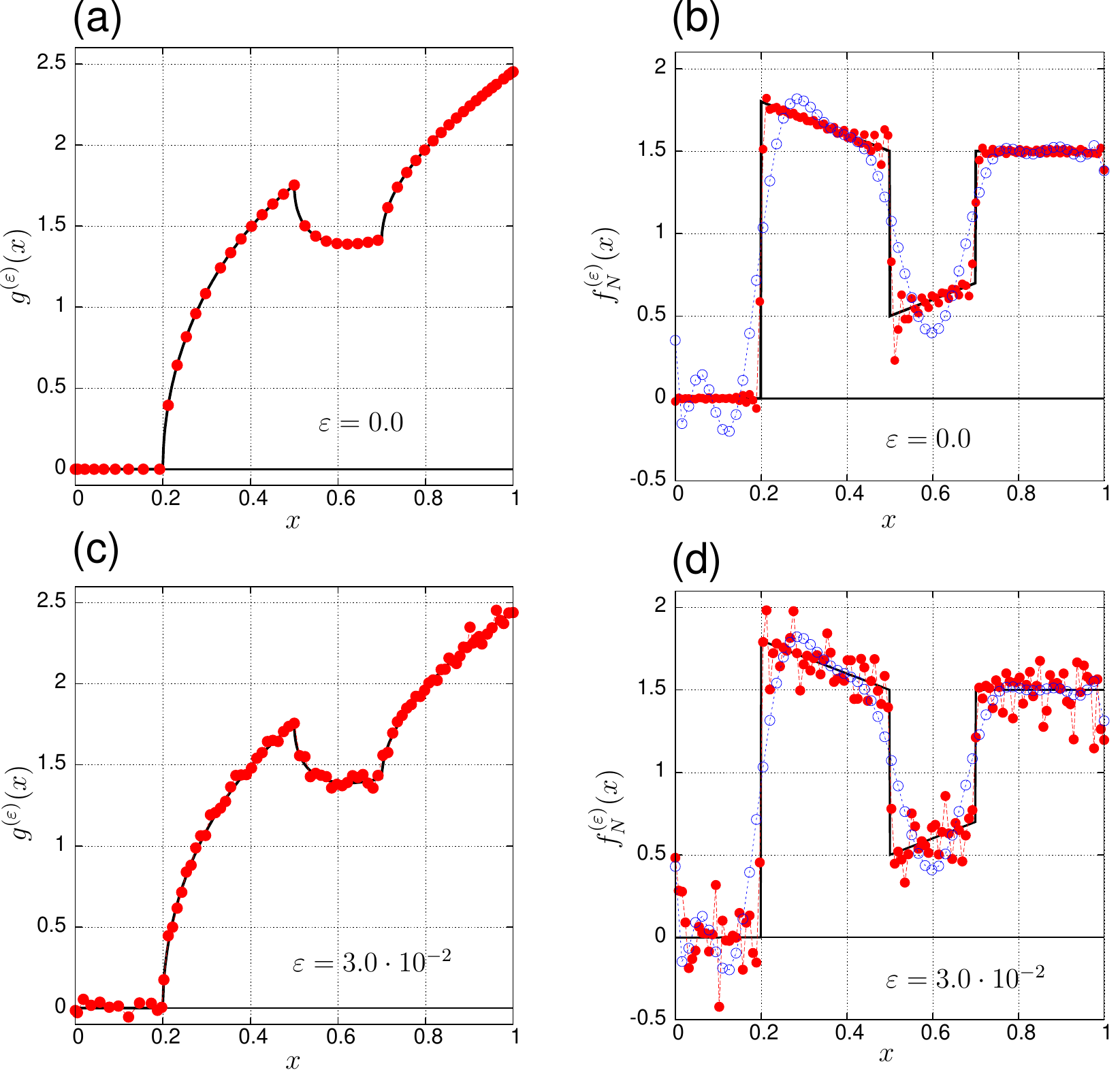} 
\caption{\label{fig:4} \small Analysis of the function $f_3(x)$ (see \eqref{Ammari2.f}).
(a): Abel transform $g_3(x)$ (solid line; see \eqref{Ammari2.g}). The filled dots represent the input data samples $g_3^{(\varepsilon)}(\sin^2 (x_j/2))$ (see \eqref{r.3}); The number of input samples is $N_s = 256$; 
$\varepsilon = 0.0$, i.e., the input data samples are affected only by roundoff error.
(b): Inverse Abel transform $f_3(x)$ (solid line) of the function $g_3(x)$. The open circles represent samples of the approximation $f_N^{(\varepsilon)}(x)$ (see \eqref{r.4}) computed from the input data shown in (a) with $N=15$. The filled circles are samples of $f_N^{(\varepsilon)}(x)$ with $N=128$.
(c): Noisy input data $g_3^{(\varepsilon)}(\sin^2 (x_j/2))$;
$\varepsilon = 3.0\cdot10^{-2}$ ($\mathrm{SNR}=22.0\,\mathrm{dB}$). 
(d) Inverse Abel transform $f_N^{(\varepsilon)}(x)$ computed from the data shown in (c) by Eq. \eqref{r.4}. Open circles: $N=15$; filled circles: $N=128$.
}
\end{figure}

Panels (a) and (b) refer to the case $\varepsilon=0.0$, whereas
panel (c) and (d) exhibit the analysis in the case of input data highly contaminated by Gaussian white noise 
($\mathrm{SNR}=22.0\,\mathrm{dB}$). Figure \ref{fig:4}(b) shows clearly that the convergence is only in the 
sense of the $L^2$-norm, with the appearance of the Gibbs phenomenon in the neighborhood of the discontinuities. 
The convergence in regions \emph{far} from the jumps is nonetheless very good. A similar behavior is displayed by 
Fig. \ref{fig:4}(d) where the inverse Abel transform $f_N^{(\varepsilon)}(x)$ is plotted with $N=15$ (open circles)
and $N=128$ (filled circles). For small values of $N$ the quality of reconstruction is comparable with the 
noiseless case of panel (b), whereas it deteriorates significantly for $N=128$ in view of the propagation
of the high level of noise affecting the input data (see Fig. \ref{fig:4}(c)).

\section{Conclusions}
\label{se:conclusions}

We have presented a new procedure for the computation of the inverse Abel transform. The solution is given 
in terms of a Legendre sum, whose coefficients are the Fourier coefficients of a suitable function associated 
with the data. The resulting algorithm is thus extremely simple and fast since the coefficients of the 
solution can be computed efficiently by means of a single FFT. This makes the algorithm particularly appropriate
for the Abel inversion of experimental data given on (even nonequispaced) 
points of the Abel transform domain since the
core of the computation, that is, the calculation of the Legendre coefficients $c_n^{(\varepsilon)}$,
can be simply performed by means of a nonuniform FFT.
The convergence of the solution has been proved and rigorous stability estimates have been given
in Propositions \ref{pro:1} and \ref{pro:2} and in Theorem \ref{the:1}
in terms of level of noise affecting the input data and of smoothness of the sought solution.
Finally, we have presented some numerical experiments which support the theoretical analysis, showing
the stability of the algorithm for the reconstruction of inverse Abel transform functions with 
different smoothness and various amounts of noise on the data.

\section*{Acknowledgement}
\noindent
This research has been partially funded by C.N.R. - Italy, Project MD.P01.004.001.

\end{document}